\documentclass[11pt, a4paper, oneside]{amsart}

\usepackage[english]{babel}
\usepackage{amsmath, amsthm, amsfonts, mathrsfs, amssymb}
\usepackage{mathtools}
\mathtoolsset{centercolon}
\usepackage{booktabs}
\usepackage[shortlabels]{enumitem}
\setlist[itemize]{leftmargin=20pt}
\usepackage[colorlinks, citecolor = blue]{hyperref}
\usepackage[hmargin=3cm,vmargin=3cm]{geometry}
\usepackage[color=green!40]{todonotes}

\usepackage{color}
\usepackage{graphicx}
\usepackage{tikz-cd}
\usetikzlibrary{arrows}

\newcommand{\R}{\ensuremath{\mathbf{R}}}

\newcommand{\mc}{\mathcal}

\DeclarePairedDelimiter\abs{\lvert}{\rvert}

\DeclarePairedDelimiter\cbrace\{\}
\DeclarePairedDelimiter\ha()
\DeclarePairedDelimiter{\ip}\langle\rangle
\DeclarePairedDelimiter{\nrm}\lVert\rVert

\newcommand{\hab}[1]{\bigl(#1\bigr)}

\newcommand{\bracb}[1]{\bigl[#1\bigr]}

\newcommand{\has}[1]{\Bigl(#1\Bigr)}

\newcommand{\bracs}[1]{\Bigl[#1\Bigr]}

\DeclareMathOperator{\loc}{loc}
\DeclareMathOperator{\supp}{supp}
\DeclareMathOperator{\ind}{\mathbf{1}}

\DeclareMathOperator*{\esssup}{ess\,sup}
\DeclareMathOperator*{\essinf}{ess\,inf}

\DeclareMathOperator{\BMO}{BMO}
\DeclareMathOperator{\VMO}{VMO}
\DeclareMathOperator{\CMO}{CMO}

\newcommand{\dd}{\hspace{2pt}\mathrm{d}}

\def\avint_#1{\mathchoice{\mathop{\kern 0.2em\vrule width 0.6em height 0.69678ex depth -0.58065ex \kern -0.8em \intop}\nolimits_{\kern -0.4em#1}}{\mathop{\kern 0.1em\vrule width 0.5em height 0.69678ex depth -0.60387ex \kern -0.6em \intop}\nolimits_{#1}} {\mathop{\kern 0.1em\vrule width 0.5em height 0.69678ex depth -0.60387ex \kern -0.6em \intop}\nolimits_{#1}} {\mathop{\kern 0.1em\vrule width 0.5em height 0.69678ex depth -0.60387ex \kern -0.6em \intop}\nolimits_{#1}}}

\newtheorem{theorem}{Theorem}
\newtheorem{corollary}[theorem]{Corollary}
\newtheorem{lemma}[theorem]{Lemma}
\newtheorem{proposition}[theorem]{Proposition}

\newtheorem{TheoremLetter}{Theorem}
{}
\newtheorem{PropositionLetter}[TheoremLetter]{Proposition}

\theoremstyle{remark}
\newtheorem{remark}[theorem]{Remark}

\theoremstyle{definition}
\newtheorem{definition}[theorem]{Definition}

\numberwithin{theorem}{section}
\numberwithin{equation}{section}
\title{Extrapolation of compactness on Banach function spaces}

\author{Emiel Lorist}
\address[E. Lorist]{Delft Institute of Applied Mathematics \\ Delft University of Technology \\ P.O. Box 5031\\ 2600 GA Delft \\The Netherlands}
\email{e.lorist@tudelft.nl}

\author{Zoe Nieraeth}
\address[Z. Nieraeth]{BCAM\textendash  Basque Center for Applied Mathematics, Bilbao, Spain}
\email{zoe.nieraeth@gmail.com}
\thanks{Z. N. is supported by the grant Juan de la Cierva formación 2021 FJC2021-046837-I, the Basque Government through the BERC 2022-2025 program, by the Spanish State Research Agency project PID2020-113156GB-I00/AEI/10.13039/501100011033 and through BCAM Severo Ochoa excellence accreditation SEV-2023-2026.}

\allowdisplaybreaks

\begin{document}

\begin{abstract}
We prove an extrapolation of compactness theorem for operators on Banach function spaces satisfying certain convexity and concavity conditions. In particular, we show that the boundedness of an operator $T$ in the weighted Lebesgue scale and the compactness of $T$ in the unweighted Lebesgue scale yields compactness of $T$ on a very general class of Banach function spaces. 
As our main new tool, we prove various characterizations of the boundedness of the Hardy-Littlewood maximal operator on such spaces and their associate spaces, using a novel sparse self-improvement technique. We apply our main results to prove compactness of the commutators of singular integral operators and pointwise multiplication by functions of vanishing mean oscillation on, for example, weighted variable Lebesgue spaces.
\end{abstract}




\maketitle

\section{Introduction}
The classical Rubio de Francia extrapolation theorem \cite{Ru82,GR85} is one of the most powerful tools in the theory of weighted norm inequalities. In its simplest form, it states that if an operator $T$ is bounded on the weighted Lebesgue space $L^p_w(\R^d)$ \textit{some} $p \in (1,\infty)$ and all weights $w \in A_p$, then $T$ is automatically bounded on $L^p_w(\R^d)$ for \textit{all} $p \in (1,\infty)$ and $w \in A_p$. Here we call a positive function $w$ a weight, we let $L^p_w(\R^d)$ denote the space of all functions $f$ such that $fw \in L^p(\R^d)$, and write $w \in A_p$ if
$$
[w]_p:= \sup_Q \,\has{\frac{1}{\abs{Q}}\int_Q {w}^p}^{\frac1p} \has{\frac{1}{\abs{Q}}\int_Q {w}^{-p'}}^{\frac1{p'}} <\infty,
$$
where the supremum is taken over all cubes $Q\subseteq \R^d$. 

In recent years, Rubio de Francia's extrapolation theorem has been extended to compact operators. Again in its simplest form,  Hyt\"onen and Lappas \cite{HL23} showed that if 
\begin{itemize}
\item $T$ is \textit{bounded}  on $L^p_w(\R^d)$ for some $p \in (1,\infty)$ and \textit{all} weights $w \in A_p$;
\item $T$ is \textit{compact} on $L^p_w(\R^d)$ for some $p \in (1,\infty)$ and \textit{some} weight $w \in A_p$;
\end{itemize}
then $T$ is compact on $L^p_w(\R^d)$ for all $p \in (1,\infty)$ and $w \in A_p$. Note that in typical applications, one checks the compactness assumption  for $p=2$ and $w=1$, i.e. on the Hilbert space $L^2(\R^d)$.

\medskip

In order to fix ideas, let us briefly sketch the proof of Hyt\"onen and Lappas \cite{HL23} in the simplest case stated above.
In essence, the argument has three main ingredients:
\begin{enumerate}[(1)]
\item\label{it:ingredient1}  The  Rubio de Francia extrapolation theorem;
\item\label{it:ingredient2} Interpolation of compactness: For Banach function spaces $X_0$ and $X_1$ and an operator $T$ which  is compact on $X_0$ and bounded on $X_1$, $T$ is compact on the  product space
$X_0^{1-\theta}\cdot X_1^\theta$ for $\theta \in (0,1)$;
\item\label{it:ingredient3} The self-improvement property of Muckenhoupt classes: For $w \in A_p$ there is an $1<r<p$ such that $w^r \in A_{p/r}$.
\end{enumerate}
Using these ingredients, one starts by observing that $T$ is bounded on $L^p_w(\R^d)$ for all $p \in (1,\infty)$ and $w \in A_p$ by \ref{it:ingredient1} and hence, by \ref{it:ingredient2}, the compactness of the operator $T$ on $L^{p}_{w}(\R^d)$ can be used to deduce the compactness of $T$  on $$L^2(\R^d) = L^{p}_{w}(\R^d)^{\frac12}\cdot L^{p'}_{w^{-1}}(\R^d)^{\frac12}.$$
Next, fix  $p\in(1,\infty)$ and $w\in A_p$. Using \ref{it:ingredient3}  on both $w\in A_p$ and $w^{-1}\in A_{p'}$, one finds  $1<r<p<s<\infty$ such that $w^r\in A_{\frac{p}{r}}$ and $w^{-s'}\in A_{\frac{p'}{s'}}$ and
$1-\frac{1}{r}=\frac{1}{s}.$
 One readily checks that this is equivalent to the condition $w_{r,s}\in A_{p_{r,s}}$, where
\[
\frac{1}{p_{r,s}}:=\frac{\frac{1}{p}-\frac{1}{s}}{\frac{1}{r}-\frac{1}{s}},\qquad \text{and} \qquad w_{r,s}:=w^{\frac{1}{\frac{1}{r}-\frac{1}{s}}}.
\]
Note that the affine transformation that maps $\frac{1}{p}\to\frac{1}{p_{r,s}}$ is the one that maps the interval $(\frac{1}{s},\frac{1}{r})$ to the interval $(0,1)$ through a translation by $-\frac{1}{s}$ and a scaling by a factor of $\frac{1}{r}-\frac{1}{s}$. Similarly, the space $X_{r,s}:=L^{p_{r,s}}_{w_{r,s}}(\R^d)$ can be scaled and translated back to the space $X:=L^p_w(\R^d)$ through the factorization
\begin{equation}\label{eq:introfactorization}
X=(X_{r,s})^{\frac{1}{r}-\frac{1}{s}}\cdot L^s(\R^d)= (X_{r,s})^{1-\theta} \cdot L^2(\R^d)^\theta, \qquad \theta = \frac2s.
\end{equation}
Therefore, since $T$ is bounded on $X_{r,s}$ and compact on $L^2(\R^d)$, this means it is also compact on $L^p_w(\R^d)$ by \ref{it:ingredient2}, proving the result.

\medskip

The extrapolation theorem of Rubio de Francia has been generalized to a general class of Banach function spaces. This was first done for rearrangement invariant spaces by Curbera, Garc\'ia-Cuerva, Martell, and P\'erez in \cite{CGMP06}. In the recent work of Cao, M\'arin and Martell \cite{CMM22} weighted Banach function spaces under weighted analogues of the conditions of \cite{BS88} were considered. In \cite{Ni23} by the second author this was generalized to the class of saturated spaces, which is also the class of Banach function spaces we consider in this work, see Section~\ref{sec:BFS}. We refer the reader to \cite[Secion~4.7]{Ni23} for a direct comparison of the assumptions with those of \cite{CMM22}.  In particular, it was shown in \cite{Ni23} that if $T$ is bounded on $L^p_w(\R^d)$ for some $p \in (1,\infty)$ and all weights $w \in A_p$, then $T$ is bounded on any Banach function space $X$ for which
\[
M:X\to X,\qquad M:X'\to X',
\]
where $M$ denotes the Hardy--Littlewood maximal operator and $X'$ the associate space of $X$. This, of course, includes the weighted Lebesgue spaces $X=L^p_w(\R^d)$ for $w \in A_p$, as it is well-known that $M$ is bounded on $X$  and $X' = L^{p'}_{w^{-1}}(\R^d)$.

Thus, we observe that \ref{it:ingredient1}  is available in the setting of Banach function spaces. Moreover, \ref{it:ingredient2} is already phrased in this setting. Therefore to extend the extrapolation of compactness theorem of Hyt\"onen and Lappas \cite{HL23} to the setting of Banach function spaces, one only needs to find a suitable replacement for \ref{it:ingredient3}, which will be the main contribution of this paper.

In the above proof sketch, \ref{it:ingredient3} was used  to deduce the factorization in \eqref{eq:introfactorization}. Thus, we are looking for a self-improvement property of the form: If $X$ is a Banach function space such that
\[
M:X\to X,\qquad M:X'\to X',
\]
then there are $1<r<s<\infty$ with $1-\frac{1}{r}=\frac{1}{s}$ such that
\begin{equation}\label{eq:selfimpintro}
    M:X_{r,s}\to X_{r,s},\quad M:(X_{r,s})'\to (X_{r,s})'
\end{equation}
for some suitable space $X_{r,s}$ satisfying \eqref{eq:introfactorization}.
In \cite{Ni23} it was shown that a space $X_{r,s}$ such that \eqref{eq:introfactorization} holds exists if and only if 
$X$ is $r$-convex and $s$-concave, i.e., 
\begin{align*}
    \big\|(|f|^r+|g|^r)^{\frac{1}{r}}\big\|_X&\leq\big(\|f\|_X^r+\|g\|_X^r\big)^{\frac{1}{r}},&& f,g\in X\\
\big(\|f\|_X^s+\|g\|_X^s\big)^{\frac{1}{s}}&\leq\big\|(|f|^s+|g|^s)^{\frac{1}{s}}\big\|_X,&& f,g\in X.
\end{align*}
We note that in this case $X_{r,s}$ is given by the formula
\begin{equation*}
X_{r,s}:= \bracs{\bracb{(X^r)'}^{(\frac{s}{r})'}}'.
\end{equation*}
Combined with the boundedness of $M$ on $X_{r,s}$ for $r$  small enough and $s$ large enough, which we will discuss below, we have sketched the proof of our first main result.

\begin{TheoremLetter}\label{thm:A}
Let 
\[
T:\bigcup_{\substack{p\in(1,\infty)\\ w\in A_p}}L^p_w(\R^d)\to L^0(\R^d)
\]
be a linear operator such that
\begin{itemize}
\item $T$ is {bounded}  on $L^p_w(\R^d)$ for some $p \in (1,\infty)$ and {all} weights $w \in A_p$;
\item $T$ is {compact} on $L^p_w(\R^d)$ for some $p \in (1,\infty)$ and {some} weight $w \in A_p$.
\end{itemize}
Let $X$ be a Banach function space over $\R^d$ such that
\[
M:X\to X,\quad M:X'\to X',
\]
and assume $X$ is  $r$-convex and $s$-concave for some $1<r<s<\infty$.
Then $T:X\to X$ is compact.
\end{TheoremLetter}

We prove Theorem~\ref{thm:A} as Theorem~\ref{thm:main} below. We note that the $r$-convexity and $s$-concavity conditions in the case of the Lebesgue space $X=L^p_w(\R^d)$ are satisfied with $r=s=p$. Theorem~\ref{thm:A} is also applicable to, for example,  weighted variable Lebesgue spaces $X=L_w^{p(\cdot)}(\R^d)$. Here the function $p \colon \R^d \to (0,\infty)$ has to satisfy
\[
1<\essinf p \leq \esssup p <\infty
\]
so that $X$ is $r$-convex and $s$-concave with $r=\essinf p$, $s=\esssup p$, and the weight and the exponent have to satisfy some additional condition ensuring the boundedness of $M$ on $X$ (see, e.g., \cite{DHHR11, Le23} for the unweighted setting and \cite{CFN12} for the weighted setting).

As we shall see in Section \ref{sec:applications}, Theorem~\ref{thm:A} can  be used to deduce the compactness of commutators of singular integral operators and multiplication by functions with vanishing mean oscillation. We refer the reader to \cite[Section 5-8, 10]{HL23} for further examples of operators to which Theorem~\ref{thm:A} is applicable.

\begin{remark}
It is clear that the proof strategy of Theorem~\ref{thm:A} cannot work without the convexity and concavity assumptions, since, as mentioned before, the existence of the factorization \eqref{eq:introfactorization} implies that $X$ is $r$-concave and $s$-concave. This means that Theorem~\ref{thm:A} is, in particular,  not applicable to Morrey spaces, as Morrey spaces are not $s$-concave for any $s<\infty$. In \cite{La22}, Lappas proved that extrapolation of compactness in the Morrey scale is possible if one replaces the compactness assumption on $L^p_w(\R^n)$ by a compactness assumption on a Morrey space. Since there are factorization results in the spirit of \eqref{eq:introfactorization}, but with $L^s(\R^d)$ replaced by this Morrey space, this allows one to follow the same lines of reasoning as before. For a general, non-convex or non-concave Banach function space $X$, it is not clear what a suitable replacement for $L^s(\R^d)$ would be.     
\end{remark}

\medskip

Let us return to the analogue of the self-improvement property of the Muckenhoupt classes needed to prove Theorem~\ref{thm:A}, which was stated in \eqref{eq:selfimpintro}. For a Banach function space $X$, it was shown by Lerner and Ombrosi  \cite{LO10} that if $M:X\to X$, then there is an $r>1$ such that also $M:X^r\to X^r$ and hence, if
\[
M:X\to X,\quad M:X'\to X'.
\]
as in Theorem~\ref{thm:A}, 
we can  find $1<r<s<\infty$ so that
\begin{equation}\label{eq:introxmbound}
    M:X^r\to X^r,\qquad M:(X')^{s'}\to (X')^{s'}.
\end{equation}
Unfortunately, it is not clear if this implies the bounds
\begin{equation}\label{eq:introxrsmbound}
M:X_{r,s}\to X_{r,s},\qquad M:(X_{r,s})'\to (X_{r,s})'.
\end{equation}
In fact, it was shown in \cite[Theorem 2.36]{Ni23} that \eqref{eq:introxrsmbound} implies \eqref{eq:introxmbound} and the converse is an open problem, see \cite[Conjecture~2.39]{Ni23}.

Instead of using the self-improvement result of \cite{LO10} to the spaces $X$ and $X'$ separately, we will prove a simultaneous self-improvement result to show that if $M$ is bounded on $X$ and $X'$, and the space $X$ is $r^*$-concave and $s^*$-concave for some $1<r^*\leq s^*<\infty$, then \eqref{eq:introxrsmbound} holds for all $1<r\leq r^*$ small enough and $s^* \leq s< \infty$ large enough. This is a direct consequence of our second result. 
\begin{PropositionLetter}\label{thm:B}
Let $r^*\in(1,\infty)$ and let $X$ be an $r^*$-convex Banach function space over $\R^d$. Then the following are equivalent:
\begin{enumerate}[(i)]
\item\label{it:introselfimp1} We have $M:X\to X$ and $M:X'\to X'$;
\item There is an $r_0\in(1,r^*]$ so that for all $r\in(1,r_0]$ we have 
\[
M:X^{{r}}\to X^{{r}},\qquad M:(X^{{r}})'\to (X^{{r}})';
\]
\item\label{it:introselfimp3} There is an $r\in(1,r^*]$ so that $M:(X^{r})'\to (X^{r})'$.
\end{enumerate}
\end{PropositionLetter}
Proposition~\ref{thm:B} is proved as Theorem~\ref{thm:mainsi} below. It relies on a sparse characterization of the boundedness of $M$ on $X$ and $X'$, followed by a sparse self-improvement result based on a novel use of the classical reverse H\"older inequality of Muckenhoupt weights. Applying Theorem \ref{thm:B} first to $X$ and then to the resulting space $(X^r)'$ yields \eqref{eq:introxrsmbound} for some $1<r<s<\infty$, see Corollary \ref{cor:si}.
We note that Proposition~\ref{thm:B} is also of independent interest, as various works use \ref{it:introselfimp3} as an assumption, often not realizing that it is equivalent to \ref{it:introselfimp1}.

\medskip

Rubio de Francia's extrapolation theorem for $L^p_w(\R^d)$ has been generalized to the off-diagonal setting by Harboure, Mac\'ias and Segovia \cite{HMS88} and to the limited range setting by Auscher and Martell \cite{AM07}. Both settings were extended to general (quasi)-Banach function spaces $X$ in \cite{Ni23}. Using \cite[Theorem A]{Ni23}, we also obtain a limited range, off-diagonal extrapolation of compactness theorem for Banach function spaces, which is our final main result.  We refer the reader to Section \ref{sec:hard} for the definition of $A_{\vec{p},(\vec{r},\vec{s})}$.
\begin{TheoremLetter}\label{thm:C}
Let $\alpha\in\R$ and let $r_1,r_2\in[1,\infty)$, $s_1,s_2\in(1,\infty]$ satisfy $r_j<s_j$ for $j\in\{1,2\}$ and
\[
\tfrac{1}{r_1}-\tfrac{1}{r_2}=\tfrac{1}{s_1}-\tfrac{1}{s_2}=\alpha.
\]
Define
\[
\mathcal{P}:=\big\{(p_1,p_2)\in(0,\infty]^2:\tfrac{1}{p_j}\in\big[\tfrac{1}{s_j},\tfrac{1}{r_j}\big],\,j\in\{1,2\},\,\tfrac{1}{p_1}-\tfrac{1}{p_2}=\alpha\big\}
\]
and let
\[
T:\bigcup_{\substack{(p_1,p_2)\in\mathcal{P}\\ w\in A_{\vec{p},(\vec{r},\vec{s})}}} L^{p_1}_w(\R^d)\to L^0(\R^d)
\]
be a linear operator such that
\begin{itemize}
\item $T$ is {bounded}  from $L^{p_1}_w(\R^d)$ to $L^{p_2}_w(\R^d)$ for some $(p_1,p_2)\in\mathcal{P}$ and all $w\in A_{\vec{p},(\vec{r},\vec{s})}$;
\item $T$ is {compact} from $L^{p_1}_w(\R^d)$ to $L^{p_2}_w(\R^d)$ for some $(p_1,p_2)\in\mathcal{P}$ and some $w\in A_{\vec{p},(\vec{r},\vec{s})}$.
\end{itemize}
Let $r_j<r_j^\ast<s_j^\ast<s_j$ and let $X_j$ be an $r_j^\ast$-convex and $s_j^\ast$-concave Banach function space for $j\in\{1,2\}$ satisfying
\[
(X_1)_{r_1,s_1}=(X_2)_{r_2,s_2}
\]
and
\[
M:(X_1)_{r_1,s_1}\to (X_1)_{r_1,s_1},\quad M:((X_1)_{r_1,s_1})'\to ((X_1)_{r_1,s_1})'.
\]
Then $T:X_1\to X_2$ is compact.
\end{TheoremLetter}
We prove Theorem~\ref{thm:C} as Theorem~\ref{thm:mainlrod} below. Note that Theorem~\ref{thm:C}  recovers \cite[Theorem 2.3 and 2.4]{HL23} for $X_1$ and $X_2$ weighted Lebesgue spaces in a unified result.

\begin{remark}
For weighted Lebesgue spaces, the result in \cite{HL23} was further generalized to multilinear operators by Cao, Olivo and Yabuta \cite{COY22} (see also \cite{HL22} by Hyt\"onen and Lappas). Currently, multilinear extrapolation in quasi-Banach function spaces has not yet been proven in its full expected generality (see \cite[Section 6]{Ni23}). Moreover, the compactness of multilinear operators on product spaces seems to be only available for bilinear operators (see \cite[Section 3]{CFM20}) and, although appearing naturally in the multilinear setting, quasi-Banach function spaces are not allowed in this result. Because of these limitations, we will not develop multilinear versions of our results in the current paper. We do point out that a bilinear compact extrapolation for products of weights classes extending \cite[Theorem~1.2]{COY22} can easily be obtained using \cite[Theorem~B]{Ni23}, but since the proof is just an application of the linear case presented in this paper on each component, we do not consider this an important contribution to the literature, and therefore omit it. 
\end{remark}

Finally, we would like to note that for weighted Lebesgue spaces, an extrapolation of compactness result has also been obtained in the two-weight setting by Liu, Wu and Yang \cite{LWY23}.

\medskip

The plan for this paper is as follows. We start in Section \ref{sec:BFS} by defining Banach function spaces and all their relevant properties. Afterwards, in Section \ref{sec:selfimp}, we prove the self-improvement property of the maximal operator on Banach function spaces stated in Theorem \ref{thm:B}. Sections \ref{sec:simple} and \ref{sec:hard} are devoted to proving the extrapolation of compactness results in the full range case (Theorem \ref{thm:A}) and limited range, off-diagonal case (Theorem \ref{thm:C}) respectively. Finally, in Section \ref{sec:applications} we apply Theorem \ref{thm:A} to deduce the compactness of commutators of both Calder\'on--Zygmund and rough homogeneous singular integral operators with pointwise multiplication by a function with vanishing mean oscillation.

\section{Banach function spaces}\label{sec:BFS}
Let $(\Omega,\mu)$ be a $\sigma$-finite  measure space. Let $L^0(\Omega)$ denote the space of measurable functions on $(\Omega,\mu)$. A vector space $X \subseteq L^0(\Omega)$ equipped with a norm $\nrm{\,\cdot\,}_X$ is called a \emph{Banach function space over $\Omega$} if it satisfies the following properties:
\begin{itemize}
  \item \textit{Ideal property:} If $f\in X$ and $g\in L^0(\Omega)$ with $|g|\leq|f|$, then $g\in X$ with $\nrm{g}_X\leq \nrm{f}_X$.
  \item \textit{Fatou property:} If $0\leq f_n \uparrow f$ for $(f_n)_{n\geq 1}$ in $X$ and $\sup_{n\geq 1}\nrm{f_n}_X<\infty$, then $f \in X$ and $\nrm{f}_X=\sup_{n\geq 1}\nrm{f_n}_X$.
  \item \textit{Saturation property:} For every measurable $E\subseteq\Omega$ of positive measure, there exists a measurable $F\subseteq E$ of positive measure with $\ind_F\in X$.
\end{itemize}
We note that the saturation property is equivalent to the assumption there is an $f \in X$ such that $f>0$ a.e. Moreover, the Fatou property ensures that $X$ is complete.

We define the associate space $X'$ as the space of all $g \in L^0(\R^n)$ such that
\begin{equation*}
\nrm{g}_{X'}:= \sup_{\nrm{f}_X \leq 1} \int_{\R^n} \abs{fg}<\infty,
\end{equation*}
which is again a Banach function space.
By the Lorentz--Luxembourg theorem we have $X''=X$ with equal norms.

For proofs and a further elaboration on these properties we refer the reader to the book of Zaanen \cite{Za67} and the recent survey \cite{LN23b}.

\begin{remark}
    Throughout the literature, following the book by Bennet and Sharpley \cite{BS88}, in the definition of a Banach function space $X$ it is often in addition assumed that for all measurable $E\subseteq \Omega$ with $\mu(E)<\infty$ one has \begin{equation}\label{illegal}
\ind_E\in X \quad \text{and} \quad \ind_E \in X'.
\end{equation} Note that this implies the saturation property. However, \eqref{illegal} is too restrictive to study weighted norm inequalities in harmonic analysis. Indeed, there are examples of weighted Lebesgue spaces $L^p_w(\R^d)$ for $p \in (1,\infty)$ and $w \in A_p$ that do not satisfy \eqref{illegal}, see \cite[Section 7.1]{SHYY17}.
\end{remark}

\subsection{Convexity properties}
Let $X$ be a Banach function space over $\Omega$ and $1 \leq p \leq q \leq \infty$. We call $X$ \emph{$p$-convex} if \begin{align*}
    \big\|(|f|^p+|g|^p)^{\frac{1}{p}}\big\|_X&\leq\big(\|f\|_X^p+\|g\|_X^p\big)^{\frac{1}{p}},&& f,g\in X,\intertext{and we call $X$ \emph{$q$-concave} if}
\big(\|f\|_X^q+\|g\|_X^q\big)^{\frac{1}{q}}&\leq\big\|(|f|^q+|g|^q)^{\frac{1}{q}}\big\|_X,&& f,g\in X.
\end{align*}
Note that any Banach function space is $1$-convex by the triangle inequality and $\infty$-concave by the ideal property. One often defines $p$-convexity and $q$-concavity using finite sums of elements from $X$ and a  constant in the defining inequalities, but, by \cite[Theorem 1.d.8]{LT79}, one can always renorm $X$ such that these constants are equal to one, yielding our definition.

We note that if $X$ is $p$-convex and $q$-concave, then $X$ is also $p_0$ convex and $q_0$-concave for all $p_0 \in [1,p]$ and $q_0 \in [q,\infty]$ and $X'$ is $q'$-convex and $p'$-concave (see, e.g., \cite[Section 1.d]{LT79}).

For $p\in (0,\infty)$ we define the \emph{$p$-concavification of $X$} by
$$
X^p:= \cbrace{f\in L^0(\Omega):\abs{f}^{\frac{1}{p}}\in X},
$$
i.e. for a positive $f \in L^0(\Omega)$ we have $f \in X$ if and only if ${f}^p \in X^p$. We equip $X^p$ with the quasi-norm
$$
\nrm{f}_{X^p}:= \nrm{\abs{f}^{\frac1p}}_X^p,\qquad f \in X^p.
$$
Note that $X$ is a Banach function space if and only if $X$ is $p\vee 1$-convex.

\begin{definition}
Let $1\leq r<s\leq \infty$ and let $X$ be an $r$-convex and $s$-concave Banach function space. We define the \emph{$(r,s)$-rescaled Banach function space of $X$} by
\begin{equation*}
X_{r,s}:= \bracs{\bracb{(X^r)'}^{(\frac{s}{r})'}}',
\end{equation*}
which is again a Banach function space.
\end{definition}

\subsection{Calder\'on--Lozanovskii products}
Let $X_0$ and $X_1$ be Banach function spaces over $\Omega$ and $\theta \in (0,1)$. We define the \emph{Calder\'on--Lozanovskii product} (see \cite{Ca64,Lo69})  $X_\theta:= X_0^{1-\theta}\cdot X_1^{\theta}$ as the space of those $h\in L^0(\Omega)$ for which there exist $0\leq f\in X_0$, $0\leq g\in X_1$ such that $|h|\leq f^{1-\theta}g^\theta$. We equip this space with the norm
\[
\|h\|_{X_\theta}:=\inf\|f\|_{X_0}^{1-\theta}\|g\|_{X_1}^\theta,
\]
where the infimum is taken over all $0\leq f\in X_0$, $0\leq g\in X_1$ for which  $|h|\leq f^{1-\theta}g^\theta$. We note that $X_\theta$ is a Banach function space, see e.g. \cite[Appendix~7]{CNS03}.

The following proposition will play a key role in the proof of Theorem \ref{thm:A} and Theorem \ref{thm:C}. Note that one can interpret the appearing products as Calder\'on--Lozanovskii products since $$L^s(\Omega) = L^{1+\frac{s}{r'}}(\Omega)^{1-(\frac1r-\frac1s)}.$$
\begin{proposition}  \label{prop:fact}Let $1\leq r<s\leq\infty$ and let $X$  Banach function space over $\Omega$.
\begin{enumerate}[(i)]
\item\label{it:lozfact1} If $X$ is $r$-convex and $s$-concave, then 
$$
X = (X_{r,s})^{\frac1r-\frac1s} \cdot L^s(\Omega).
$$
\item\label{it:lozfact2}  If there is a Banach function space $Y$ such that
$$
X = Y^{\frac1r-\frac1s} \cdot L^s(\Omega),
$$
then $X$ is $r$-convex and $s$-concave and $Y=X_{r,s}$.
\end{enumerate}
\end{proposition}
\begin{proof}
   For \ref{it:lozfact1} we refer the reader to \cite[Corollary 2.12]{Ni23}, and \ref{it:lozfact2} is proven analogously to \cite[Theorem 2.13]{Ni23}, substituting $X_{r,s}$ by $Y$.
\end{proof}

\section{Sparse self-improvement for the maximal operator}\label{sec:selfimp}
As we will need Proposition~\ref{thm:B} in the proof of Theorem~\ref{thm:A}, we start in this section by proving 
Proposition~\ref{thm:B} and its consequences. First we introduce some notation. For $r \in (0,\infty)$, a cube $Q \subseteq \R^d$ and function $f \in L^0(\R^d)$ we define the $r$-average of $f$ by
$
  \ip{f}_{r,Q}:= \ha{\frac{1}{\abs{Q}} \int_Q \abs{f}^r}^{1/r},
$
and we define the Hardy--Littlewood maximal operator by
\begin{align*}
  Mf&:= \sup_{Q} \,\ip{f}_{1,Q} \ind_Q,
\end{align*}
where the supremum is taken over all cubes $Q$ in $\R^d$.

\begin{theorem}\label{thm:mainsi}
Let $r^*\in(1,\infty)$ and let $X$ be an $r^*$-convex Banach function space over $\R^d$. Then the following are equivalent:
\begin{enumerate}[(i)]
\item\label{it:Ainftyconvexbfs1} We have $M:X\to X$ and $M:X'\to X'$;
\item\label{it:Ainftyconvexbfs2} There is an $r_0\in(1,r^*]$ so that for all ${r}\in(1,r_0]$ we have 
\[
M:X^{{r}}\to X^{{r}},\qquad M:(X^r)'\to (X^r)';
\]
\item\label{it:Ainftyconvexbfs3} There is an $r\in(1,r^*]$ so that $M:(X^{r})'\to (X^{r})'$.
\end{enumerate}
\end{theorem}
Before we turn to the proof, we provide two useful corollaries. First of all, we obtain the following equivalent formulations of the bounds
\[
M:X\to X,\qquad M:X'\to X',
\]
assumed in Theorem~\ref{thm:A}.
\begin{corollary}\label{cor:mboundequivalences}
Let $X$ be a Banach function space for which there exist $1<r^*<s^*<\infty$ such that $X$ is $r^*$-convex and $s^*$-concave. Then the following are equivalent:
\begin{enumerate}[(i)]
\item\label{it:convexbfs1} We have $M:X\to X$ and $M:X'\to X'$;
\item\label{it:convexbfs2} There is an $r\in(1,r^*]$ so that $M:(X^{r})'\to (X^{r})'$;
\item\label{it:convexbfs3} There is an $s\in[s^*,\infty)$ so that $M:X_{1,s}\to X_{1,s}$.
\end{enumerate}
\end{corollary}
\begin{proof}
The equivalence of \ref{it:convexbfs1} and \ref{it:convexbfs2} is contained in Theorem~\ref{thm:mainsi} whereas the equivalence of \ref{it:convexbfs1} and \ref{it:convexbfs3} follows from applying the first equivalence with $X$ replaced by $X'$, which is $(s^*)'$-convex, to find an $s'\in(1,(s^*)']$ for which $M$ is bounded on $[(X')^{s'}]'=X_{1,s}$.
\end{proof}

Our second corollary is the self-improvement property in Banach function spaces discussed in the introduction, which will replace the self-improvement property of the Muckenhoupt classes in the proof of Theorem~\ref{thm:A} and Theorem~\ref{thm:C}:
\begin{corollary}\label{cor:si}
Let $X$ be a Banach function space for which there exist $1<r^*<s^*<\infty$ such that $X$ is $r^*$-convex, $s^*$-concave, and
\[
M:X\to X,\qquad M:X'\to X'.
\]
Then there exist $r_0\in(1,r^*]$, $s_0\in[s^*,\infty)$ such that for all ${r}\in(1,r_0]$ and ${s}\in[s_0,\infty)$ we have
\[
M:X_{{r},{s}}\to X_{{r},{s}},\qquad M:(X_{{r},{s}})'\to(X_{{r},{s}})'.
\]
\end{corollary}
\begin{proof}
By Theorem~\ref{thm:mainsi} there is a $r_0\in(1,r^*]$ for which
\[
M: X^{r_0}\to X^{r_0},\qquad M:(X^{r_0})'\to (X^{r_0})'.
\]
Note that $(X^{r_0})'$ is $(\frac{s^*}{r_0})'$-convex. By applying Theorem~\ref{thm:mainsi} to $(X^{r_0})'$, we find a $t_0\in(1,(\frac{s^*}{r_0})']$ such that $M$ is bounded on $[(X^{r_0})']^{t_0}$ and $\big([(X^{r_0})']^{t_0}\big)'$.

Now define $s_0:=r_0t_0'>s^*$. Then we have
$t_0=(\frac{s_0}{r_0})'$,
so that
\[
M:X_{r_0,s_0}\to X_{r_0,s_0},\qquad M:(X_{r_0,s_0})'\to(X_{r_0,s_0})'.
\]
Letting ${r}\in(1,r_0)$, ${s}\in(s_0,\infty)$ and noting that $(X_{r_0,s_0})'=(X')_{s_0',r_0'}$ by \cite[Proposition~2.14]{Ni23}, it follows from \cite[Proposition~2.30]{Ni23} that also
\[
M:X_{{r},{s}}\to X_{{r},{s}},\qquad M:(X_{{r},{s}})'\to(X_{{r},{s}})'.
\]
The assertion follows.
\end{proof}

We now turn to the proof of Theorem~\ref{thm:mainsi}. As a final preparation, we require a lemma on sparse operators. A collection of cubes $\mathcal{S}$ in $\R^d$ is called \emph{sparse} if for each $Q\in\mathcal{S}$ there is a measurable set $E_Q\subseteq Q$ for which $|E_Q|\geq\frac12|Q|$ and, furthermore, the collection $(E_Q)_{Q\in\mathcal{S}}$ is pairwise disjoint. For a sparse collection of cubes $\mathcal{S}$ and $f\in L^0(\R^d)$ we define
\[
T_{\mathcal{S}}f:=\sum_{Q\in\mathcal{S}}\langle f\rangle_{1,Q}\ind_Q.
\]
We have the following characterization of when $M$ is bounded on $X$ and $X'$ in terms of $T_{\mathcal{S}}$.
\begin{lemma}\label{lem:sparsembounds}
Let $X$ be a Banach function space over $\R^d$. Then
\[
M:X\to X,\qquad M:X'\to X'
\]
if and only if we have $T_{\mathcal{S}}:X\to X$ for all sparse collections of cubes $\mathcal{S}$ with
\[
\sup_{\mathcal{S}\text{ is sparse}}\|T_{\mathcal{S}}\|_{X\to X}<\infty.
\]
\end{lemma}
\begin{proof}
For $f,g\in L^0(\R^d)$ we define
\[
M_{1,1}(f,g):=\sup_{Q}\,\langle f\rangle_{1,Q}\langle g\rangle_{1,Q}\ind_Q.
\]
By \cite[Proposition~6.1]{Ni23} it suffices to show that $M_{1,1}:X\times X'\to L^1(\R^d)$ if and only if for all sparse collections $\mathcal{S}$ we have $T_\mathcal{S}:X\to X$ uniformly in $\mathcal{S}$ with
\[
\sup_{\mathcal{S}\text{ is sparse}}\|T_{\mathcal{S}}\|_{X\to X}\eqsim_{d}\|M_{1,1}\|_{X\times X'\to L^1(\R^d)}.
\]

Assume first that $M_{1,1}:X\times X'\to L^1(\R^d)$. If $\mathcal{S}$ is sparse, then for $f \in X$ and $g \in X'$
\begin{align*}
\|(T_\mathcal{S} f)g\|_{L^1(\R^d)}&=\sum_{Q\in\mathcal{S}}\langle f\rangle_{1,Q}\langle g\rangle_{1,Q}|Q|\\
&\leq 2\sum_{Q\in\mathcal{S}}\int_{E_Q}\!M_{1,1}(f,g)\,\mathrm{d}x\\
&\leq 2\,\|M_{1,1}(f,g)\|_{L^1(\R^d)}
\end{align*}
Hence,
\[
\sup_{\mathcal{S}\text{ is sparse}}\|T_{\mathcal{S}}\|_{X\to X}
=\sup_{\mathcal{S}\text{ is sparse}}\sup_{\substack{\|f\|_X=1\\ \|g\|_{X'}=1}}\|(T_\mathcal{S} f)g\|_{L^1(\R^d)}\leq 2\,\|M_{1,1}\|_{X\times X'\to L^1(\R^d)}.
\]
Conversely, using \cite[Proposition~3.2.10]{Ni20}, for each $f,g\in L^0(\R^d)$, each dyadic lattice $\mathcal{D}$ in $\R^d$, and each finite collection $\mathcal{F}\subseteq\mathcal{D}$, there exists a sparse collection $\mathcal{S}\subseteq\mathcal{F}$ such that
\[
M^\mathcal{F}_{1,1}(f,g)\leq 4\sum_{Q\in\mathcal{S}}\langle f\rangle_{1,Q}\langle g\rangle_{1,Q}\ind_Q,
\]
where the superscript $\mc{F}$ indicates that the defining supremum in the definition of $M_{1,1}$ is only taken over $Q \in \mc{F}$. Hence,
\begin{align*}
\|M^\mathcal{F}_{1,1}(f,g)\|_{L^1(\R^d)}
&\leq 4\sum_{Q\in\mathcal{S}}\langle f\rangle_{1,Q}\langle g\rangle_{1,Q}|Q|=4\,\|(T_\mathcal{S} f)g\|_{L^1(\R^d)}\\
&\leq 4\,\sup_{\mathcal{S}\text{ is sparse}}\|T_{\mathcal{S}}\|_{X\to X}\|f\|_X\|g\|_{X'}.
\end{align*}
The assertion now follows from the monotone convergence theorem and the $3^d$ lattice theorem.
\end{proof}

\begin{proof}[Proof of Theorem~\ref{thm:mainsi}]
For \ref{it:Ainftyconvexbfs1}$\Rightarrow$\ref{it:Ainftyconvexbfs2}, we first note that the sharp reverse H\"older inequality \cite[Theorem 2.3]{HP13} states that there is a dimensional constants $c_d\geq 1$ so that for all weights $w$ such that
$$
[w]_{A_\infty} := \sup_Q \frac{1}{w(Q)} \int_Q M(w\ind_Q) <\infty
$$
and all ${r}\in(1,\infty)$ satisfying ${r}'\geq c_d[w]_{A_\infty}$ we have
\[
\langle w\rangle_{{r},Q}\leq 2\, \langle w\rangle_{1,Q}
\]
for all cubes $Q$ in $\R^d$.

Fix an $r_0\in(1,r^*]$ with $r_0'\geq 2c_d\|M\|_{X\to X}$, and let ${r}\in(1,r_0]$. For $f\in X^r$ we define
\[
w:=\sum_{k=0}^\infty\frac{M^k(|f|^{\frac{1}{{r}}})}{2^k\|M\|^k_{X\to X}},
\]
where $M^k$ denotes the $k$-th iterate of $M$. Then we have
\[
c_d[w]_{A_\infty}\leq 2c_d\|M\|_{X\to X}\leq {r}',
\]
and therefore $\langle w\rangle_{{r},Q}\leq 2\langle w\rangle_{1,Q}$ for all cubes $Q$. 

Let $\mathcal{S}$ be a sparse collection. Since $|f|^{\frac{1}{{r}}}\leq w$ and $\|w\|_X\leq 2\|f\|_{X^{{r}}}^{\frac{1}{ {r}}}$, we obtain
\begin{align*}
\|T_{\mathcal{S}}f\|_{X^{ {r}}}
&=\Big\|\Big(\sum_{Q\in\mathcal{S}}\langle |f|^{\frac{1}{ {r}}}\rangle_{ {r},Q}^{ {r}}\ind_Q\Big)^{\frac{1}{ {r}}}\Big\|^{ {r}}_{X}
\leq\Big\|\sum_{Q\in\mathcal{S}}\langle w\rangle_{ {r},Q}\ind_Q\Big\|^{ {r}}_{X}\\
&\leq 2^r \|T_{ {\mathcal{S}}}w\|_X^{ {r}}
\leq 2^r \|T_{\mathcal{S}}\|_{X\to X}^{ {r}}\|w\|^{ {r}}_{X}\leq 4^{ {r}}\|T_{\mathcal{S}}\|_{X\to X}^{ {r}}\|f\|_{X^{ {r}}}.
\end{align*}
Thus, we have
\[
\sup_{\mathcal{S}\text{ is sparse}}\|T_{{S}}\|_{X^{ {r}}\to X^{ {r}}}\leq 4^{r}\Big(\sup_{\mathcal{S}\text{ is sparse}}\|T_{\mathcal{S}}\|_{X\to X}\Big)^{ {r}}.
\]
The result now follows from Lemma~\ref{lem:sparsembounds}.

The implication \ref{it:Ainftyconvexbfs2}$\Rightarrow$\ref{it:Ainftyconvexbfs3} is immediate.
To prove \ref{it:Ainftyconvexbfs3}$\Rightarrow$\ref{it:Ainftyconvexbfs1}, we apply \cite[Theorem~4.16 and Remark~4.17]{Ni23} to the classical bound $$\|M\|_{L^p_w(\R^d)\to L^p_w(\R^d)}\lesssim_d p'[w]_p^{p'}$$ by Buckley \cite{Bu93} to conclude that
\[
\|M\|_{X\to X}\lesssim_d r'2^{\frac{1}{r-1}}\|M\|_{(X^{r})'\to (X^{r})'}^{\frac{1}{r-1}},\quad
\|M\|_{X'\to X'}\lesssim_d r\|M\|_{(X^{r})'\to (X^{r})'}.
\]
This proves the result.
\end{proof}

\begin{remark}
    Theorem \ref{thm:mainsi} can easily be extended to Banach function spaces $X$ over a space of homogeneous type $(S,d,\mu)$ in the sense of Coifman and Weiss \cite{CW71}. Maximal operator bounds and sparse domination estimates in this setting are available through the use of e.g. Hyt\"onen--Kairema cubes \cite{HK12} (see also Christ \cite{Ch90}) and the sharp reverse H\"older inequality needs to be replaced by the weak sharp reverse H\"older inequality by Hyt\"onen, Perez and Rela \cite{HPR12}.
\end{remark}

\section{Full range compact extrapolation}\label{sec:simple}
This section is dedicated to the proof of the full range extrapolation of compactness theorem, Theorem~\ref{thm:A}:
\begin{theorem}\label{thm:main}
Let 
\[
T:\bigcup_{\substack{p\in(1,\infty)\\ w\in A_p}}L^p_w(\R^d)\to L^0(\R^d)
\]
be a linear operator such that
\begin{itemize}
\item $T$ is {bounded}  on $L^p_w(\R^d)$ for some $p \in (1,\infty)$ and {all} weights $w \in A_p$;
\item $T$ is {compact} on $L^p_w(\R^d)$ for some $p \in (1,\infty)$ and {some} weight $w \in A_p$.
\end{itemize}
Let $X$ be a Banach function space over $\R^d$ such that
\[
M:X\to X,\quad M:X'\to X',
\]
and assume $X$ is  $r^*$-convex and $s^*$-concave for some $1<r^\ast<s^\ast<\infty$.
Then $T:X\to X$ is compact.
\end{theorem}

As noted in the introduction, we will need three main ingredients to prove Theorem \ref{thm:main}. We will need the Rubio de Francia extrapolation result from \cite{Ni23}, the self-improvement result from Section \ref{sec:selfimp} and a result on the compactness of operators on product spaces. The latter is a special case of a result of Cobos, K\"uhn and Schonbek \cite[Theorem~3.1]{CKS92} with the function parameter $\rho(t)=t^\theta$, which we formulate next.
\begin{proposition}[{\cite[Theorem~3.1]{CKS92}}]\label{prop:compactinterpolation}
Let $(\Omega,\mu)$ be a $\sigma$-finite measure space and let $X_0$, $X_1$, $Y_0$, $Y_1$ be Banach function spaces over $\Omega$. Let $T\colon X_0+X_1\to Y_0+Y_1$ be a linear operator such that
\begin{itemize}
    \item $T$ is bounded from $X_0$ to $Y_0$;
    \item $T$ is compact from $X_1$ to $Y_1$.
\end{itemize}
Then
\[
T:X_0^{1-\theta} \cdot X_1^\theta\to Y_0^{1-\theta} \cdot Y_1^\theta
\]
is  compact for all $\theta \in (0,1)$.
\end{proposition}

Having all main ingredients at our disposal, the proof of Theorem~\ref{thm:main} is rather short.

\begin{proof}[Proof of Theorem~\ref{thm:main}]
Note that, by the classical Rubio de Francia extrapolation theorem, $T$ is bounded on $L^p_w(\R^d)$ for all $p \in (1,\infty)$ and $w \in A_p$. Moreover, since 
$$L^2(\R^d) = L^{p}_{w}(\R^d)^{\frac12}\cdot L^{p'}_{w^{-1}}(\R^d)^{\frac12}$$
and $w \in A_p$ if and only if $w^{-1} \in A_{p'}$,
Proposition \ref{prop:compactinterpolation} implies that $T$ is compact on $L^2(\R^d)$.

By Corollary~\ref{cor:si}, there are $r_0\in(1,r^*]$ and $s_0\in[s^*,\infty)$ such that $M$ is bounded on $X_{{r},{s}}$ and $(X_{{r},{s}})'$ for all ${r}\in(1,r_0]$ and ${s}\in[s_0,\infty)$. Hence, by Rubio de Francia extrapolation in Banach function spaces as in \cite[Theorem~A]{Ni23}, $T$ is bounded on $X_{{r},{s}}$ for all ${r}\in(1,r_0]$ and ${s}\in[s_0,\infty)$.

By Proposition \ref{prop:fact}, we have
\[
X=X_{{r},{s}}^{1-\theta}\cdot L^p(\R^d)^\theta.
\]
with $\theta=1-(\frac{1}{{r}}-\frac{1}{{s}})\in(0,1)$ and $p=1+\frac{{s}}{{r}'}$. Choosing $\frac{1}{{r}'}=\frac{1}{{s}}$ small enough, we have $p=2$ and $T$ is bounded on $X_{{r},{s}}$.
Since $T$ is  compact on $L^2(\R^d)$, $T$ is compact on $X$ by Proposition~\ref{prop:compactinterpolation}. This proves the result.
\end{proof}

\section{Limited range, off-diagonal compact extrapolation}\label{sec:hard}
In this section, we prove Theorem~\ref{thm:C}. Essentially, the steps are the same as in the proof of Theorem~\ref{thm:main}, and the new difficulties lie mainly in unwinding the definitions while incorporating the additional parameters. 

For $1 \leq r <p<s\leq \infty$, we say that a weight $w$ belongs to the limited range Muckenhoupt class $A_{p,(r,s)}$ if 
$$
[w]_{p,(r,s)} := \sup_Q \,\ip{w}_{\frac{1}{\frac{1}{p}-\frac{1}{s}},Q}\ip{w^{-1}}_{\frac{1}{\frac{1}{r}-\frac{1}{p}},Q}< \infty.
$$
For $\alpha\in\R$, $r_1,r_2\in[1,\infty)$, $s_1,s_2\in(1,\infty]$ for which $\frac{1}{s_j}<\frac{1}{r_j}$ for $j\in\{1,2\}$ and
\[
\tfrac{1}{r_1}-\tfrac{1}{r_2}=\tfrac{1}{s_1}-\tfrac{1}{s_2}=\alpha,
\]
we note that we have $A_{p_1,(r_1,s_1)}=A_{p_2,(r_2,s_2)}$. For a weight $w$ in this class we will write $w\in A_{\vec{p},(\vec{r},\vec{s})}$. Recall that our limited range, off-diagonal extrapolation of compactness theorem reads as follows:
 
\begin{theorem}\label{thm:mainlrod}
Let $\alpha\in\R$ and let $r_1,r_2\in[1,\infty)$, $s_1,s_2\in(1,\infty]$ satisfy $\frac{1}{s_j}<\frac{1}{r_j}$ for $j\in\{1,2\}$ and
\[
\tfrac{1}{r_1}-\tfrac{1}{r_2}=\tfrac{1}{s_1}-\tfrac{1}{s_2}=\alpha.
\]
Define
\[
\mathcal{P}:=\big\{(p_1,p_2)\in(0,\infty]^2:\tfrac{1}{p_j}\in\big[\tfrac{1}{s_j},\tfrac{1}{r_j}\big],\,j\in\{1,2\},\,\tfrac{1}{p_1}-\tfrac{1}{p_2}=\alpha\big\}
\]
and let
\[
T:\bigcup_{\substack{(p_1,p_2)\in\mathcal{P}\\ w\in A_{\vec{p},(\vec{r},\vec{s})}}} L^{p_1}_w(\R^d)\to L^0(\R^d)
\]
be a linear operator such that
\begin{itemize}
\item $T$ is {bounded}  from $L^{p_1}_w(\R^d)$ to $L^{p_2}_w(\R^d)$ for some $(p_1,p_2)\in\mathcal{P}$ and all $w\in A_{\vec{p},(\vec{r},\vec{s})}$;
\item $T$ is {compact} from $L^{p_1}_w(\R^d)$ to $L^{p_2}_w(\R^d)$ for some $(p_1,p_2)\in\mathcal{P}$ and some $w\in A_{\vec{p},(\vec{r},\vec{s})}$.
\end{itemize}
Let $r_j<r_j^\ast<s_j^\ast<s_j$ and let $X_j$ be an $r_j^\ast$-convex and $s_j^\ast$-concave Banach function space for $j\in\{1,2\}$ satisfying
\[
(X_1)_{r_1,s_1}=(X_2)_{r_2,s_2}
\]
and
\[
M:(X_1)_{r_1,s_1}\to (X_1)_{r_1,s_1},\quad M:((X_1)_{r_1,s_1})'\to ((X_1)_{r_1,s_1})'.
\]
Then $T:X_1\to X_2$ is compact.
\end{theorem}

\begin{remark}
   The limited range, off-diagonal Rubio de Francia extrapolation theorem in \cite{Ni23} is phrased for \emph{quasi}-Banach function spaces. However, one of our other main ingredients, the compactness result in Proposition~\ref{prop:compactinterpolation}, does not seem to be available in the quasi-setting. Therefore, we have stated Theorem \ref{thm:mainlrod} in the Banach function space setting and leave its extension to the quasi-Banach function space setting with $r_1,r_2 \in (0,\infty)$  as an open problem. Note that quasi-Banach function spaces typically only show up in harmonic analysis in multilinear or endpoint settings.
\end{remark}

As said, the main challenge in the proof of Theorem \ref{thm:mainlrod} is to unpack all definitions and to keep track of the additional parameters. In order to do so, we prove a couple of technical lemmata. We start with a limited range version of the self-improvement property for the maximal operator.

\begin{lemma}\label{lem:lrsi}
Let $1\leq r<r^\ast<s^\ast<s\leq\infty$ and let $X$ be an $r^\ast$-convex and $s^\ast$-concave Banach function space. If
\[
M: X_{r,s}\to X_{r,s},\quad M:(X_{r,s})'\to (X_{r,s})',
\]
then there are $r_0\in(r,r^\ast]$, $s_0\in[s^\ast,s)$ so that for all $\widetilde{r}\in(r,r_0)$, $\widetilde{s}\in(s_0,s)$ we have
\[
M:X_{\widetilde{r},\widetilde{s}}\to X_{\widetilde{r},\widetilde{s}},\quad M:(X_{\widetilde{r},\widetilde{s}})'\to (X_{\widetilde{r},\widetilde{s}})'.
\]
\end{lemma}
\begin{proof}
By Theorem~\ref{thm:mainsi} there is a $q_0\in(1,r^\ast]$ so that for all $q\in(1,q_0]$ we have
\[
M: X_{r,s}^q\to X_{r,s}^q,\quad M:(X_{r,s}^q)'\to (X_{r,s}^q)'.
\]
Defining
\[
\tfrac{1}{r_0}:=\tfrac{1}{q}\big(\tfrac{1}{r}-\tfrac{1}{s}\big)+ \tfrac{1}{s}\in(\tfrac{1}{s},\tfrac{1}{r}),
\]
it follows from \cite[Lemma~2.32]{Ni23} that for $q>1$ chosen small enough so that $r_0\leq r^\ast$, we have
\[
X_{r,s}^q=X_{r_0,s}.
\]
Note that $(X_{r_0,s})'=[(X^{r_0})']^{(\frac{s}{r_0})'}$ is $t^*$-convex, where 
$$
t^* : = \frac{\frac{1}{r_0}-\frac{1}{s}}{\frac{1}{r_0}-\frac{1}{s^\ast}}.
$$
Thus, by  Theorem~\ref{thm:mainsi}, we find a $t_0\in(1,t^*]$ so that for all $t\in(1,t_0]$ we have that $M$ is bounded on $[(X^{r_0})']^{t(\frac{s}{r_0})'}$ and $\big([(X^{r_0})']^{t(\frac{s}{r_0})'}\big)'$.

Now define
\[
\tfrac{1}{s_0}:=\tfrac{1}{r_0}-\tfrac{1}{t}\big(\tfrac{1}{r_0}-\tfrac{1}{s}\big)=\tfrac{1}{q}\tfrac{1}{t'}\big(\tfrac{1}{r}-\tfrac{1}{s}\big)+\tfrac{1}{s}\in(\tfrac{1}{s},\tfrac{1}{r_0})
\]
which, if $t>1$ is chosen small enough, satisfies $s_0\geq s^\ast$. Then we have
\[
t\big(\tfrac{s}{r_0}\big)'=\big(\tfrac{s_0}{r_0}\big)'
\]
so that
\[
M:X_{r_0,s_0}\to X_{r_0,s_0},\qquad M:(X_{r_0,s_0})'\to(X_{r_0,s_0})'.
\]
Now let $\widetilde{r}\in(r,r_0)$, $\widetilde{s}\in(s_0,s)$. Noting that  $(X_{r_0,s_0})'=((X^r)')_{(\frac{s_0}{r})',(\frac{r_0}{r})'}$ by \cite[Proposition~2.14]{Ni23}, it follows from \cite[Proposition~2.30]{Ni23} that also
\[
M:X_{\widetilde{r},\widetilde{s}}\to X_{\widetilde{r},\widetilde{s}},\qquad M:(X_{\widetilde{r},\widetilde{s}})'\to(X_{\widetilde{r},\widetilde{s}})'.
\]
The assertion follows.
\end{proof}

Next, we prove a rescaling lemma.

\begin{lemma}\label{lem:lrrescale}
Let $1\leq r<{\widetilde{r}}<\widetilde{s}<s\leq\infty$ 
and let $X$ be an ${\widetilde{r}}$-convex and $\widetilde{s}$-concave Banach function space. Define
\[
\frac{1}{t}:=\frac{\frac{1}{r}\frac{1}{\widetilde{s}}-\frac{1}{{\widetilde{r}}}\frac{1}{s}}{\frac{1}{r}-\frac{1}{s}}
\]
Then $X$ is $t$-concave, and $[X_{{\widetilde{r}},t}]^{\frac{1}{r}}$ is $r$-convex and $s$-concave with
\[
([X_{{\widetilde{r}},t}]^{\frac{1}{r}})_{r,s}=X_{{\widetilde{r}},\widetilde{s}}.
\]
\end{lemma}
\begin{proof}
We have
\[
\frac{1}{t}=\frac{\frac{1}{r}(\frac{1}{\widetilde{s}}-\frac{1}{s})+(\frac{1}{r}-\frac{1}{{\widetilde{r}}})\frac{1}{s}}{\frac{1}{r}-\frac{1}{s}}>0
\]
and
\[
\frac{1}{\widetilde{s}}-\frac{1}{t}=\frac{1}{s}\frac{\frac{1}{{\widetilde{r}}}-\frac{1}{\widetilde{s}}}{\frac{1}{r}-\frac{1}{s}}\geq 0.
\]
Thus, $t\geq \widetilde{s}$ and, hence, $X$ is $t$-concave. Therefore $([X_{{\widetilde{r}},t}]^{\frac{1}{r}})^r=X_{{\widetilde{r}},t}$ is a Banach function space, and hence, $[X_{{\widetilde{r}},t}]^{\frac{1}{r}}$ is $r$-convex. Moreover, since
\[
\big(\tfrac{t}{{\widetilde{r}}}\big)'\big(\tfrac{s}{r}\big)'=\big(\tfrac{\widetilde{s}}{{\widetilde{r}}}\big)'
\]
we have
\[
\big([([X_{{\widetilde{r}},t}]^{\frac{1}{r}})^r]'\big)^{(\frac{r}{s})'}=[(X^{{\widetilde{r}}})']^{(\frac{{\widetilde{r}}}{\widetilde{s}})'}=(X_{\widetilde{r},\widetilde{s}})'
\]
which is a Banach function space, since $X$ is ${\widetilde{r}}$-convex and $\widetilde{s}$-concave. This proves that $[X_{{\widetilde{r}},t}]^{\frac{1}{r}}$ is $s$-concave. Moreover, taking associate spaces in this equality, the final assertion follows.
\end{proof}

We finish our preparation for the proof of Theorem \ref{thm:mainlrod} with a factorization lemma in the limited range setting.

\begin{lemma}\label{lem:lrinterpolation}
Let $1 \leq r<\widetilde{r}<\widetilde{s}<s\leq\infty$ and let $X$ be an $\widetilde{r}$-convex and $\widetilde{s}$-concave Banach function space. Define
\[
\frac{1}{t}:=\frac{\frac{1}{r}\frac{1}{\widetilde{s}}-\frac{1}{\widetilde{r}}\frac{1}{s}}{\frac{1}{r}-\frac{1}{s}}
\]
and
\[
\theta:=1-\frac{\frac{1}{\widetilde{r}}-\frac{1}{\widetilde{s}}}{\frac{1}{r}-\frac{1}{s}}\in(0,1),\qquad p=\frac{\frac{1}{\widetilde{s}}-\frac{1}{s}+\frac{1}{r}-\frac{1}{\widetilde{r}}}{\frac{1}{r}(\frac{1}{\widetilde{s}}-\frac{1}{s})+(\frac{1}{r}-\frac{1}{\widetilde{r}})\frac{1}{s}}\in (1,\infty).
\]
Then
\[
X=\big([X_{\widetilde{r},t}]^{\frac{1}{r}}\big)^{1-\theta}\cdot L^p(\R^d)^\theta.
\]
\end{lemma}
\begin{proof}
Since
\[
\tfrac{1}{r}(1-\theta)=\tfrac{1}{\widetilde{r}}-\tfrac{1}{t}
\]
and $t=\frac{p}{\theta}$, we have by Proposition \ref{prop:fact} and Lemma \ref{lem:lrrescale}
\[
\big([X_{\widetilde{r},t}]^{\frac{1}{r}}\big)^{1-\theta}\cdot L^p(\R^d)^\theta
=(X_{\widetilde{r},t})^{\frac{1}{\widetilde{r}}-\frac{1}{t}}\cdot L^{t}(\R^d)=X,
\]
as asserted.
\end{proof}

\begin{proof}[Proof of Theorem~\ref{thm:mainlrod}]
Let $(p_1,p_2)\in\mc{P}$ and $w\in A_{\vec{p},(\vec{r},\vec{s})}$ so that $T:L^{p_1}_w(\R^d)\to L^{p_2}_w(\R^d)$ is compact. Note that, by limited range, off-diagonal Rubio de Francia extrapolation for weighted Lebesgue spaces as in \cite[Theorem~1.8]{CM17} (which is also a special case of \cite[Theorem~A]{Ni23}), we actually obtain the boundedness assumption on $T$ for all $(q_1,q_2) \in \mc{P}$ and weights in $A_{\vec{q},(\vec{r},\vec{s})}$. In particular, this is the case for $(q_1,q_2)\in\mc{P}$ defined by 
\[
\frac{1}{q_j}:=\frac{1}{s_j}+\frac{1}{r_j}-\frac{1}{p_j}, \qquad j\in\{1,2\},
\]
and the weight $w^{-1}\in A_{\vec{q},(\vec{r},\vec{s})}$, which one can directly verify using the definition of the weight constant. Since
\[
L^{\frac{1}{\frac{1}{2}(\frac{1}{r_j}+\frac{1}{s_j})}}(\R^d)=L^{p_j}_w(\R^d)^{\frac{1}{2}}\cdot L^{q_j}_{w^{-1}}(\R^d)^{\frac{1}{2}},
\]
it follows Proposition~\ref{prop:compactinterpolation} that $T:L^{\frac{1}{\frac{1}{2}(\frac{1}{r_1}+\frac{1}{s_1})}}(\R^d)\to L^{\frac{1}{\frac{1}{2}(\frac{1}{r_2}+\frac{1}{s_2})}}(\R^d)$ is compact. Thus, we have reduced the compactness assumption to the case $\frac{1}{p_j}=\frac{1}{2}(\frac{1}{r_j}+\frac{1}{s_j})$ for $j\in\{1,2\}$ and $w=1$.

Next, by Lemma~\ref{lem:lrsi} we have
\[
M:(X_1)_{\widetilde{r}_1,\widetilde{s}_1}\to (X_1)_{\widetilde{r}_1,\widetilde{s}_1},\quad M:((X_1)_{\widetilde{r}_1,\widetilde{s}_1})'\to ((X_1)_{\widetilde{r}_1,\widetilde{s}_1})'
\]
for all $\widetilde{r}_1\in(r_1,r_1^\ast]$, $\widetilde{s}_1\in[s_1^\ast,s_1)$ with
\[
\tfrac{1}{\widetilde{s}_1}-\tfrac{1}{s_1}=\tfrac{1}{r_1}-\tfrac{1}{\widetilde{r}_1}=\varepsilon
\]
for $\varepsilon>0$ small enough.

Defining $\frac{1}{\widetilde{r}_2}:=\frac{1}{\widetilde{r}_1}-\alpha$ and $\frac{1}{\widetilde{s}_2}:=\frac{1}{\widetilde{s}_1}-\alpha$, it follows from \cite[Lemma~2.32]{Ni23} that for
\[
\beta:=\frac{\frac{1}{r_1}-\frac{1}{s_1}}{\frac{1}{\widetilde{r}_1}-\frac{1}{\widetilde{s}_1}}=\frac{\frac{1}{r_2}-\frac{1}{s_2}}{\frac{1}{\widetilde{r}_2}-\frac{1}{\widetilde{s}_2}},\quad\gamma:=\frac{\frac{1}{\widetilde{r}_1}-\frac{1}{r_1}}{\frac{1}{\widetilde{r}_1}-\frac{1}{r_1}+\frac{1}{s_1}-\frac{1}{\widetilde{s}_1}}=\frac{\frac{1}{\widetilde{r}_2}-\frac{1}{r_2}}{\frac{1}{\widetilde{r}_2}-\frac{1}{r_2}+\frac{1}{s_2}-\frac{1}{\widetilde{s}_2}}
\]
we have
\begin{align*}
(X_1)_{\widetilde{r_1},\widetilde{s_1}}
&=(X_1)_{r_1,s_1}^\beta\cdot L^{\frac{1}{1-\gamma}}(\R^d)^{1-\beta}\\
&=(X_2)_{r_2,s_2}^\beta\cdot L^{\frac{1}{1-\gamma}}(\R^d)^{1-\beta}=(X_2)_{\widetilde{r}_2,\widetilde{s}_2}.
\end{align*}
Hence, by Lemma~\ref{lem:lrrescale} we conclude that
\[
\big([(X_1)_{\widetilde{r}_1,t_1}]^{\frac{1}{r_1}}\big)_{r_1,s_1}=(X_1)_{\widetilde{r_1},\widetilde{s_1}} = (X_2)_{\widetilde{r}_2,\widetilde{s}_2}=\big([(X_2)_{\widetilde{r}_2,t_2}]^{\frac{1}{r_2}}\big)_{r_2,s_2},
\]
where
\[
\frac{1}{t_j}=\frac{\frac{1}{r_j}\frac{1}{\widetilde{s}_j}-\frac{1}{\widetilde{r}_j}\frac{1}{s_j}}{\frac{1}{r_j}-\frac{1}{s_j}}, \qquad  j\in\{1,2\}.
\]
Thus, by \cite[Theorem~A]{Ni23} it follows that
\[
T:[(X_1)_{\widetilde{r}_1,t_1}]^{\frac{1}{r_1}}\to[(X_2)_{\widetilde{r}_2,t_2}]^{\frac{1}{r_2}}.
\]
 Moreover, by Lemma~\ref{lem:lrinterpolation} we have
\begin{align*}
    X_1&=\big([(X_1)_{\widetilde{r}_1,t_1}]^{\frac{1}{r_1}}\big)^{1-\theta}\cdot L^{p_1}(\R^d)^\theta,\\X_2&=\big([(X_2)_{\widetilde{r}_2,t_2}]^{\frac{1}{r_2}}\big)^{1-\theta}\cdot L^{p_2}(\R^d)^\theta
\end{align*}
with
\[
\theta=1-\frac{\frac{1}{\widetilde{r}_1}-\frac{1}{\widetilde{s}_1}}{\frac{1}{r_1}-\frac{1}{s_1}}=1-\frac{\frac{1}{\widetilde{r}_2}-\frac{1}{\widetilde{s}_2}}{\frac{1}{r_2}-\frac{1}{s_2}},
\]
and
\[
\tfrac{1}{p_j}=\tfrac{1}{2}\big(\tfrac{1}{r_j}+\tfrac{1}{s_j}\big), \qquad j\in\{1,2\}.
\]
 Hence, it follows from Proposition~\ref{prop:compactinterpolation} that $T:X_1\to X_2$ is compact. This proves the assertion.
\end{proof}

\section{Applications}\label{sec:applications}
In this final section, we will briefly outline applications of Theorem~\ref{thm:A}  and Theorem~\ref{thm:C} to the compactness of commutators of singular integral operators and multiplication by functions with vanishing mean oscillation. We refer the reader to \cite[Section 5-10]{HL23} for further examples of operators to which Theorem~\ref{thm:A} or Theorem~\ref{thm:C} are applicable.

We start with an application to Calder\'on--Zygmund singular integral operators. Let $T \colon L^2(\R^d) \to L^2(\R^d)$ be a bounded linear operator and suppose that, for any $f \in C^\infty_c(\R^d)$, $T$ has the representation
$$
Tf(x) = \int_{\R^d}K(x,y)f(y)\dd y, \qquad x \in \R^d \setminus \supp(f),
$$
where the kernel satisfies the estimates
\begin{align*}
    \abs{K(x,y) - K(x,y)}   &\leq \omega\hab{\tfrac{\abs{x-x'}}{\abs{x-y}}}, && 0<\abs{x-x'} <\tfrac12 \abs{x-y},\\
    \abs{K(x,y) - K(x,y')}   &\leq \omega\hab{\tfrac{\abs{x-x'}}{\abs{x-y}}}, && 0<\abs{y-y'} <\tfrac12 \abs{x-y},
\end{align*}
for some increasing, subadditive $\omega\colon [0,1] \to [0,\infty)$. If  $\int_0^1\omega(t) \frac{\rm{d} t}{t}<\infty$, we call $T$ a \emph{Calder\'on--Zygmund operator with Dini-continuous kernel}.

We will be concerned with the commutators 
$$
[b,T]f := bT(f) -T(bf)
$$
for pointwise multipliers $b \in L^1_{\loc}(\R^d)$ belonging to the space
$$
\CMO := \overline{C_c^\infty(\R^d)}^{\nrm{\cdot}_{\BMO}},
$$
where $\BMO$ denotes the classical space of functions with bounded mean oscillation. Note that the space $\CMO$ is the space of all functions with vanishing mean oscillation and is therefore  also  denoted by $\VMO$ in the literature, see, e.g., \cite[Lemma 3]{Uc78}.

\begin{theorem}\label{thm:CZO}
Let $T$ be a Calder\'on--Zygmund  operator with {Dini-}con\-tinuous kernel. Let $1<r<s<\infty$ and suppose $X$ is an $r$-convex and $s$-concave Banach function space over $\R^d$ such that
\[
M:X\to X,\quad M:X'\to X'.
\]
For $b \in \mathrm{CMO}$, the commutator $[b,T]\colon X \to X$ is a compact operator.
\end{theorem}

\begin{proof}
    In view of Theorem \ref{thm:A}, it suffices to check that $[b,T]$ is bounded on $L^2_w(\R^d)$ for all $w \in A_2$ and that $[b,T]$ is compact on $L^2(\R^d)$.  For the weighted boundedness it suffices to check that $T$ is bounded on $L^2_w(\R^d)$ for all $w \in A_2$ by \cite[Theorem 2.13]{ABKP93} (see \cite{BMMST20} for a modern approach). The boundedness of $T$ on $L^2_w(\R^d)$ for all $w \in A_2$ is classical (see \cite{La17,LO19} for a modern approach). Finally, the compactness of $[b,T]$ on $L^2(\R^d)$ was shown by Uchiyama \cite{Uc78}.
\end{proof}

Theorem \ref{thm:CZO} in the specific case where $X$ is a weighted Lebesgue space was previously obtained by Clop and Cruz \cite{CC13} and subsequently proven using an extrapolation of compactness argument by Hyt\"onen and Lappas \cite{HL23}. 
Theorem \ref{thm:CZO} yields many examples of Banach function spaces $X$ for which the compactness of $[b,T]$ was previously unknown.

For example, one can consider weighted variable Lebesgue spaces $X=L_w^{p(\cdot)}(\R^d)$. Here the function $p \colon \R^d \to (0,\infty)$ has to satisfy
\begin{equation}\label{eq:varlebesgueassumption1}
1<\essinf p \leq \esssup p <\infty
\end{equation}
so that we can take $r=\essinf p$, $s=\esssup p$ in Theorem~\ref{thm:A}, and some additional condition ensuring the boundedness of $M$ on $X$ (see, e.g., \cite{DHHR11, Le23} for the unweighted setting and \cite{CFN12} for the weighted setting). As a matter of fact, it was originally shown by Diening in \cite[Theorem~8.1]{Di05} that in the case of unweighted variable Lebesgue spaces satisfying \eqref{eq:varlebesgueassumption1} we have $M:X\to X$ if and only if $M:X'\to X'$. We also refer the reader to \cite{Le23} for an explicit characterization of when $M$ is bounded on $X$ in terms of the exponent function $p$. Diening's duality result was extended to weighted variable Lebesgue spaces $X=L^{p(\cdot)}_w(\R^d)$ by Lerner in \cite{Le16b}. In conclusion, in the setting of weighted variable Lebesgue spaces, the boundedness condition of $M$ in Theorem~\ref{thm:A} only needs to be checked on $X$, as the boundedness on $X'$ is then guaranteed.

We refer the reader to \cite{Ni23} for further examples of Banach function spaces satisfying the conditions in Theorem \ref{thm:CZO}.
We note that for the case that $X$ is a Morrey space, similar results were obtained by Arai and Nakai \cite{AN19} and Lappas \cite{La22}, which lie beyond the scope of our general framework as explained in the introduction.

\bigskip

Next, let us consider rough singular integral operators
$$
T_\Omega f:= {\rm p.v.}\int_{\R^d} \frac{\Omega(x-y)}{\abs{x-y}^d} f(y)\dd y, \qquad x \in \R^d,
$$
where $\Omega\in L^r(\mathbf{S}^{d-1})$ for some $r \in [1,\infty]$ is homogeneous of order zero and has mean value zero.

\begin{theorem}\label{thm:RSIO}
Let $1\leq r<r^*<s^* <\infty$ and let $\Omega \in L^{r'}(\Omega)$. Suppose $X$ is an $r^*$-convex and $s^*$-concave Banach function space over $\R^d$ such that
\[
M:X^r\to X^r,\quad M:(X^r)'\to (X^r)'.
\]
For $b \in \mathrm{CMO}$, the commutator $[b,T_\Omega]\colon X \to X$ is a compact operator.
\end{theorem}

\begin{proof}
    In view of Theorem \ref{thm:C}, it suffices to check that $[b,T]$ is bounded on $L^p_w(\R^d)$ for some $p \in (r,\infty)$ and all $w \in A_{p,(r,\infty)}$ and that $[b,T]$ is compact on $L^p(\R^d)$. As in the proof of Theorem \ref{thm:CZO}, for the weighted boundedness it suffices note that $T_\Omega$ is bounded on $L^p_w(\R^d)$ for all $w \in A_{p,(r,\infty)}$, which was shown independently by Watson \cite{Wa90} and Duoandikoetxea \cite{Du93}. The compactness of $[b,T_\Omega]$ on $L^p(\R^d)$ was shown by Chen and Hu \cite{CH15}.
\end{proof}

Theorem \ref{thm:RSIO} in the specific case where $X$ is a weighted Lebesgue space was previously obtained by Guo and Hu \cite{GH16} and subsequently proven using an extrapolation of compactness argument by Hyt\"onen and Lappas \cite{HL23}. In the case that $X$ is a Morrey space, similar results were also obtained in \cite{GH16} and \cite{La22}.

To conclude our applications section, we note that in the recent work by Tao, Yang, Yuan and Zhang \cite{TYYZ23}, compactness of  $[b,T_\Omega]$ for $b \in \CMO$ was studied in the setting of Banach function spaces as well. Let us give a  comparison between Theorem~\ref{thm:RSIO} in the specific case $r=1$ and \cite[Theorem 3.1]{TYYZ23}
\begin{itemize}
    \item We work with Banach function spaces as in Section \ref{sec:BFS}, whereas \cite{TYYZ23} uses ball Banach function spaces, the former being a more general in the sense that any ball Banach function space is also a Banach function space in the sense of Section \ref{sec:BFS}. However, since we assume $M$ to be bounded on $X$ and $X'$, so in particular $\ind_B\in X$ and $\ind_B \in X'$ for all balls $B \subseteq \R^d$, we also have that any Banach function space satisfying the assumptions of Theorem~\ref{thm:RSIO} is a ball Banach function space.  
    \item In \cite{TYYZ23} it is assumed that there is an $r^*>1$ such that $X^{r^*}$ is a Banach function space (i.e. $X$ is $r^*$-convex) and $M$ is bounded on $(X^{r^*})'$. This is equivalent to $M$ being bounded on $X$ and $X'$ by Theorem \ref{thm:B}. Thus, the only difference in the assumptions on $X$ between Theorem \ref{thm:RSIO} with $r=1$ and \cite[Theorem 3.1]{TYYZ23} is that Theorem \ref{thm:RSIO} in addition assumes $X$ to be $s$-concave for some $s<\infty$. This, as discussed before, is a necessary assumption in our general approach.
    \item In Theorem \ref{thm:RSIO} with $r=1$  it is assumed that $\Omega \in L^\infty(\mathbb{S}^{d-1})$, whereas in \cite[Theorem 3.1]{TYYZ23} a much stronger assumption on $\Omega$ is imposed. Indeed, $\Omega$ is assumed to satisfy a Dini continuity condition, see \cite[Definition 2.15]{TYYZ23}. 
    \item The proof of Theorem \ref{thm:RSIO} uses ``soft'' techniques, whereas \cite{TYYZ23} takes a more hands-on and technical approach, developing and using a Frechet--Kolmogorov compactness criterion in ball Banach function spaces (see also \cite{GZ20}).
\end{itemize}

\bibliography{compactextrapolation}
\bibliographystyle{alpha}
\end{document}